\documentclass[12pt, a4paper, reqno]{amsart}
\usepackage{latex_base}
\usepackage{macros}
\usepackage{tablefootnote}

\newcommand{\tableTraverseTime}{
\begin{tabular}{c|ccccccc}
	$n\setminus t$ & 6 & 9 & 12 & 20 & 24 & 30 & 50\\
	\hline
	2 & 2.53 & 3.33 & 4.98 & 6.79 & 7.72 & 9.41 & 14.44\\
	\hline
	3 & 1.55 & 1.93 & 2.73 & 5.27 & 7.15 & 10.18 & 22.38\\
	\hline
	4 & 2.97 & 2.71 & 3.83 & 8.96 & 11.63 & 16.05 & 37.45\\
	\hline
	7 & $\times$ & \fewInstances{0.03} & \fewInstances{2.01} & \fewInstances{50.84} & $\times$ & $\times$ & $\times$\\
	\hline
	8 & $\times$ & 0.02 & 11.46 & 37.12 & 63.72 & 111.77 & 333.31\\
	\hline
\end{tabular}}
\newcommand{\tableForkTime}{
\begin{tabular}{c|ccccccc}
	$n\setminus t$ & 6 & 9 & 12 & 20 & 24 & 30 & 50\\
	\hline
	2 & 0.22 & 0.31 & 0.42 & 0.76 & 0.98 & 1.40 & 3.33\\
	\hline
	3 & 0.21 & 0.30 & 0.44 & 1.00 & 1.39 & 2.06 & 5.45\\
	\hline
	4 & 0.18 & 0.28 & 0.42 & 1.49 & 2.15 & 3.48 & 9.48\\
	\hline
	7 & $\times$ & \fewInstances{0.16} & \fewInstances{0.31} & \fewInstances{1.48} & $\times$ & $\times$ & $\times$\\
	\hline
	8 & $\times$ & 0.06 & 0.28 & 1.01 & 5.05 & 18.39 & 95.76\\
	\hline
\end{tabular}}
\newcommand{\Myterms}{6,9,12,20,24,30,50}
\newcommand{\variables}{2,3,4,7,8}
\newcommand{\instancesBnB}{9639}
\newcommand{\totalTimeBnB}{193.19}
\newcommand{\smallGap}{9069}
\newcommand{\smallGapPercent}{94.1}
\newcommand{\traverseQuotMin}{0.056}
\newcommand{\traverseQuotMax}{172.01}
\newcommand{\traverseQuotGM}{1.141}
\newcommand{\traverseQuotPercent}{14.1}
\newcommand{\boundForkFail}{102}

\usepackage{listings}
\setlist[description]{leftmargin=1em,labelindent=0em}

\makeatletter
\def\namedlabel#1#2{\begingroup
		#2%
		\def\@currentlabel{#2}%
		\phantomsection\label{#1}\endgroup
}
\makeatother

\author{Henning Seidler} 

\address{Henning Seidler, Technische Universit\"at Berlin, FG Security in Telecommunications, Ernst-Reuter-Platz 7, 10587 Berlin, Germany\medskip}

\email{henning.seidler@tu-berlin.de}

\subjclass[2010]{Primary: 14P99, 90-04, 90C22, 90C26; Secondary: 14Q20, 52B20, 68Q25 \textit{ACM Subject Classification:} Mathematical software performance}

\keywords{Certificate, Nonnegativity, Polynomial optimisation, Sum of nonnegative circuit polynomials, Sum of squares, Unconstrained}

\newif\ifcomment
\commentfalse
\commenttrue
\newcommand{\POEMhomepage}{https://www.user.tu-berlin.de/henning.seidler/POEM/}

\title{Improved Lower Bounds for Global Polynomial Optimisation}

\begin{document}

\pagestyle{plain}
\pagenumbering{arabic}

\begin{abstract}
	We present a branch-and-bound algorithm to improve the lower bounds obtained by SONC/SAGE.
	The running time is fixed-parameter tractable in the number of variables.
	Furthermore, we describe a new heuristic to obtain a candidate for the global minimum of a multivariate polynomial, based on its SONC decomposition.
	Applying this approach to thousands of test cases, we mostly obtain small duality gaps. 
	In particular, we optimally solve the global minimisation problem in about 70\% of the investigated cases.
\end{abstract}

\maketitle

\section{Introduction}

Finding the global minimum of a given multivariate polynomial is a well-known problem in optimisation.
This problem has countless applications, see e.g., \cite{Lasserre:Book:CPOPApplications}.
Closely connected is the decision problem, whether a given multivariate polynomial is nonnegative.
Already this problem is known to be \coNP{}-hard, as follows from \cite[Theorem 3]{Murty:Kabadi:NPC}.
Therefore, a common approach to \struc{certify nonnegativity} is to use some sufficient criterion.
The most famous approach is \struc{sums of squares (SOS)}, which dates back to Hilbert.
This approach has been widely applied with success in recent years; see e.g., \cite{Blekherman:Parrilo:Thomas,Laurent:Survey,Lasserre:Book:CPOPApplications,Lasserre:Book:CPOPIntroduction} for an overview.

However, the SOS approach has some serious drawbacks.
In 2006, Blekherman proved that for fixed even degree $d \geq 4$ and $n \to \infty$ almost every nonnegative polynomial is not SOS \cite{Blekherman:ConeComparison}.
Furthermore, deciding whether an $n$-variate polynomial of degree $d$ is SOS translates into an SDP of size $\binom{n+d}{d}$, which quickly becomes infeasible even to state, let alone be solved.
For \struc{sparse} polynomials, i.e. where the support is significantly smaller than all $\binom{n+d}{d}$ possible monomials, this is particularly critical, as it presents an exponential blow-up.
In this setting, Renegar \cite{Renegar:PSPACE} presented a more efficient algorithm which runs in polynomial space and single exponential time.
Even under the view of parametrised complexity, deciding SOS for sparse polynomial only is in \XP{} parametrised by either the degree $d$ or the number of variables $n$.

An alternative certificate of nonnegativity is based on \struc{sums of nonnegative circuit polynomials (SONC)}, introduced by Iliman and de Wolff \cite{Iliman:deWolff:Circuits}.
In a recent paper \cite{Seidler:deWolff:POEM}, we presented an algorithmic approach to obtain lower bounds via SONC, including the software POEM (Effective Methods in Polynomial Optimisation) \cite{poem:software}.
This method computes a lower bound by solving a geometric programme.
While this approach fared well in experiments, it had some major drawbacks.
\begin{enumerate}
	\item It did not include a method to find any (local) minimiser, so we could not tell the optimality gap.
	\item To find the bound via SONC, we had to perform a relaxation, that allowed to restrict on the positive orthant, which possibly worsened the results.
	\item The method could only find \emph{some} lower bound and not even the best bound theoretically obtainable via SONC.
\end{enumerate}
An important improvement on the third issue came by Chandrasekaran and Shah by using \struc{sums of arithmetic geometric mean exponentials (SAGE)}, using relative entropy programmes (REP) \cite{Chandrasekaran:Shah:SAGE,Chandrasekaran:Shah:REP}.
However, their variable substitution corresponds to a restriction on the positive orthant as well, but there they efficiently compute the best bound obtainable by SAGE.

In this contribution, we address the first two of the above issues.
First, in \cref{theorem:circuit_minimiser}, we generalise \cite{Iliman:deWolff:Circuits} to explicitly compute minimisers for arbitrary circuit polynomials.
This serves as base for our heuristic to compute the minimum of the given polynomial in \Cref{subsection:minima}.
Second, we present a branch-and-bound approach, where we branch over the signs on the variables in \Cref{section:branch_and_bound}.
This eventually gives us additional information for the sign of the terms, which allows to improve the lower bounds.
The running time may increase by a factor $2^n$, which is fixed-parameter tractable, so it is still considered efficient in parametrised complexity.
In fact, we only have to perform our initial algorithm on some of the $2^n$ orthants.
In \Cref{section:minimal_orthants} we suggest an alternative to the branch-and-bound, where we determine these orthants and compute a lower bound for each one, which can easily be done in parallel.
The time to find these orthants is negligible to the overall time.

We implemented these algorithms and included them in our software POEM.
In addition, we perform a large scale experiment on a subset of the examples from \cite{Seidler:deWolff:POEM} and present the evaluation in \Cref{section:experimental_results}.
These polynomials have up to 8 variables, degree 60 and 50 terms.
In total, our experiment contains \instancesBnB{} instances with a total running time of more than 8 days.
Overall, we observe a significant improvement of the lower bounds and about 70\% of our instances we solved optimally.

\subsection*{Acknowledgements}
We thank Timo de Wolff for his suggestion to work on a branch-and-bound approach for SONC.
Further thanks go to Helena Müller for her improved computation of minimisers of circuit polynomials.

\section{Preliminaries}
\label{section:preliminaries}

In this section we introduce our basic notation, sums of squares, sums of nonnegative circuit polynomials, and geometric programmes.

\subsection{Representing Sparse Polynomials}

Throughout the paper, we use bold letters for vectors (small) and matrices (capital), e.g., $\struc{\Vector{x}}=(x_1,\ldots,x_n) \in \mathds R^n$.
Let $\struc{\mathds R_{\geq 0}}$ and $\struc{\mathds R_{> 0}}$ denote the set of nonnegative and positive real numbers, respectively.
Furthermore, let \struc{$\mathds R[\Vector{x}] = \mathds R[x_1,\ldots,x_n]$} be the ring of real $n$-variate polynomials. 
We denote the set of all $n$-variate polynomials of degree less than or equal to $2d$ by $\struc{\mathds R[\Vector{x}]_{n,2d}}$. 
For $p \in \mathds R[\Vector{x}]$ we denote the \struc{\textit{total degree}} of $p$ by $\struc{\deg(p)}$.

We investigate \struc{\emph{sparse polynomials}} $p \in \mathds R[\Vector{x}]$ supported on a finite set $\struc{A} \subset \mathds N^n$; we write $\struc{\support{p}}$ if a clarification is necessary.
Thus, $p$ is of the form $\struc{p(\Vector{x})} = \sum_{\Vector{\alpha} \in A}^{} b_{\Vector{\alpha}}\Vector{x}^{\Vector{\alpha}}$ with $\struc{b_{\Vector{\alpha}}} \in \mathds R\setminus\{0\}$ and $\struc{\Vector{x}^{\Vector{\alpha}}} = x_1^{\alpha_1} \cdots x_n^{\alpha_n}$. 
While a multivariate polynomial may have up to $\binom{n+d}{d}$ terms, sparsity means $|A| \ll \binom{n+d}{d}$.
Unless stated differently, we follow the convention $\struc{t} = |A|$.
The support of $p$ can be expressed as an $n \times t$ matrix, which we denote by $\struc{\Matrix{A}}$, such that the $j$-th column of $\Matrix{A}$ is $\Vector{\alpha(j)}$.
Hence, $p$ is uniquely described by the pair $(\Matrix{A},\Vector{b})$, written $p = \struc{\polynomial{\Matrix{A}}{\Vector{b}}}$.

The \struc{Newton polytope} of $p$, denoted $\newton{p} = \operatorname{chull}(\support{})$, is the convex hull of all exponent vectors.
A lattice point $\Vector{\alpha}$ is called \struc{\textit{even}} if it is in $(2\mathds N)^n$ and a term $b_{\Vector{\alpha}}\Vector{x}^{\Vector{\alpha}}$ is called a \struc{\emph{monomial square}} if $b_{\Vector{\alpha}} > 0$ and $\Vector{\alpha}$ even.
We define
\begin{align*}
	\struc{\monoSquares{p}} = \left\{\Vector{\alpha} \in \support{p} : \Vector{\alpha} \in (2\mathds N)^n, b_{\Vector{\alpha}} > 0\right\}
\end{align*}
as the set of monomial squares in the support of $p$.
Moreover, we use the notation $\struc{\nonSquares{p}} = \support{p} \setminus \monoSquares{p}$ for all elements of the support of $p$, which are not a monomial square. 

\subsection{Sums of Nonnegative Circuit Polynomials}

We introduce the fundamental facts of SONC polynomials, which we use in this article. 
SONCs are constructed by \textit{circuit polynomials}; which were first introduced in \cite{Iliman:deWolff:Circuits}:

\begin{definition}
	\label{definition:CircuitPolynomial}
	A \struc{\emph{circuit polynomial}} $p = \polynomial{\Matrix{A}}{\Vector{b}} \in \mathds R[\Vector{x}]$ is of the form
	\begin{align}
		\struc{p(\Vector{x})} & = \sum_{j=0}^r b_{\Vector{\alpha}(j)} \Vector{x}^{\Vector{\alpha}(j)} + b_{\Vector{\beta}} \Vector{x}^{\Vector{\beta}}, 
		\label{Equ:CircuitPolynomial}
	\end{align}
	with $0 \leq \struc{r} \leq n$, coefficients $\struc{b_{\Vector{\alpha}(j)}} \in \mathds R_{> 0}$, $\struc{b_{\Vector{\beta}}} \in \mathds R$, exponents $\struc{\Vector{\alpha}(j)} \in (2\mathds Z)^n$, $\struc{\Vector{\beta}} \in \mathds Z^n$, such that the following condition holds:
	There exist unique, positive \struc{\emph{barycentric coordinates} $\lambda_j$} relative to the $\Vector{\alpha}(j)$ with $j=0,\ldots,r$ satisfying
	\begin{align}
		\label{equ:BarycentricCoordinates}
		\Vector{\beta} &= \sum_{j=0}^r \lambda_j \Vector{\alpha}(j) \ \text{ with } \ \lambda_j > 0 \ \text{ and } \ \sum_{j=0}^r \lambda_j = 1.
	\end{align}
	For every circuit polynomial $p$ we define the corresponding \struc{\textit{circuit number}} as
	\begin{align*}
		\struc{\Theta_p} &= \prod_{j = 0}^r \left(\frac{b_{\Vector{\alpha}(j)}}{\lambda_j}\right)^{\lambda_j}. 
		\qedhere
	\end{align*}
\end{definition}
	Condition \eqref{equ:BarycentricCoordinates} implies that $\Matrix{A}(p)$ forms a minimal affine dependent set. 
	Those sets are called \struc{\textit{circuits}}, see e.g., \cite{Oxley:MatroidTheory}.
	More specifically, Condition \eqref{equ:BarycentricCoordinates} yields that $\newton{p}$ is a simplex with even vertices $\Vector{\alpha}(0), \Vector{\alpha}(1),\ldots,\Vector{\alpha}(r)$ and that the exponent $\Vector{\beta}$ is in the relative interior of $\newton{p}$. 
	Therefore, we call the terms $p_{\Vector{\alpha}(0)} \Vector{x}^{\Vector{\alpha}(0)},\ldots,p_{\Vector{\alpha}(r)} \Vector{x}^{\Vector{\alpha}(r)}$ the \struc{\emph{outer terms}} and $p_{\Vector{\beta}} \Vector{x}^{\Vector{\beta}}$ the \struc{\emph{inner term}} of $p$. 

Circuit polynomials are proper building blocks for nonnegativity certificates since the circuit number alone determines whether they are nonnegative.

\begin{theorem}[\cite{Iliman:deWolff:Circuits}, Theorem 3.8]
	Let $p$ be a circuit polynomial of the form \eqref{Equ:CircuitPolynomial}. Then $p$ is nonnegative if and only if:
	\begin{enumerate}
		\item $p$ is a sum of monomial squares, or
		\item the coefficient $b_{\Vector{\beta}}$ of the inner term of $p$ satisfies $|b_{\Vector{\beta}}| \leq \Theta_p$.
	\end{enumerate}
	\label{thm:CircuitPolynomialNonnegativity}
\end{theorem}

To compute $\Theta_p$, we solve a system of linear equations.
Hence, we have an easily checked arithmetic condition for the nonnegativity of a circuit polynomial.
These nonnegative circuit polynomials now generate the cone, we use as our certificate of nonnegativity.
\begin{definition}
	\label{definition:SONC}
	We define for every $n,d \in \mathds N$ the set of \struc{\emph{sums of nonnegative circuit polynomials} (SONC)} in $n$ variables of degree $2d$ as
	\begin{align*}
		\struc{C_{n,2d}} = \left\{f \in \mathds R[\Vector{x}]_{n,2d} : f = \sum_{\text{finite}} p_i, \quad p_i \text{ is a nonnegative circuit polynomial} \right\}.
		\tag*{\qedhere}
	\end{align*}
\end{definition}

We denote by \SONC{} both the set of SONC polynomials and the property of a polynomial to be a sum of nonnegative circuit polynomials.

For further details about the SONC cone see \cite{deWolff:Circuits:OWR,Iliman:deWolff:Circuits, Dressler:Iliman:deWolff:Positivstellensatz}.

\subsection{Lower Bounds via SONC}
\label{subsection:SONC_algo}

Given an arbitrary polynomial $p$, we apply this approach to compute a lower bound for its values.
If we find some $\boundSONC \in \mathds R$ such that $p-\boundSONC \in \SONC$, then we have $p(x) \geq \boundSONC$ for all $x\in\mathds R^n$.
Note, that in general, this is not the infimum.

We shortly describe our algorithm from \cite{Seidler:deWolff:POEM}.
Further details can be found there.

Every monomial, that is not a square, must appear as the inner term of a circuit polynomial.
This corresponds to relaxing $p$ to the polynomial $\relaxPoly{p}$, where every non-square is equipped with a negative sign.
Furthermore, we now can restrict ourselves to the positive orthant, since $\relaxPoly{p}$ attains its minimum there.
For simplicity, we assume $p = \relaxPoly{p}$ when computing lower bounds for a polynomial.

Next, we determine the circuits involved in the decomposition.
For each $\Vector{\beta} \in \nonSquares{p}$, we write it as a convex combination of monomial squares, which means we find a solution of the LP
\begin{align*}
	\sum_{\Vector{\alpha} \in \monoSquares{p}} \lambda_{\Vector{\alpha}}\Vector{\alpha} &= \Vector{\beta} &
	\sum_{\Vector{\alpha} \in \monoSquares{p}} \lambda_{\Vector{\alpha}} &= 1 &
	\lambda_{\Vector{\alpha}} &\geq 0 \text{ for all $\Vector{\alpha} \in \monoSquares{p}$}
\end{align*}
If necessary, we further eliminate some if the $\lambda_{\Vector{\alpha}}$ until $\cover^{\Vector{\beta}} := \left\{\Vector{\alpha} : \lambda_{\Vector{\alpha}} > 0\right\}\cup\{\Vector{\beta}\}$ forms a circuit.
This yields the \struc{covering} $\cover := \left\{\cover^{\Vector{\beta}} : \Vector{\beta} \in \nonSquares{p}\right\}$.
Finally, we solve the following \struc{Geometric Programme}:
\begin{maxi}
	{\Matrix{X}}
	{b_{\Vector{0}} - \sum_{\Vector{\beta}\in\nonSquares{p}} X_{\Vector{\beta}, \Vector{0}}}
	{\label{problem:SONC}\tag{SONC}}
	{\boundSONC \ = \ }
	\myCons{\sum_{\Vector{\beta}\in\nonSquares{p}} X_{\Vector{\beta},\Vector{\alpha}}}{\leq b_{\Vector{\alpha}}}{\Vector{\alpha} \in \monoSquares{p}, \Vector{\alpha} \neq \Vector{0}}
	\myCons{\prod_{\Vector{\alpha}\in\cover^{\Vector{\beta}}} \left(\frac{X_{\Vector{\beta},\Vector{\alpha}}}{\lambda^{\Vector{\beta}}_{\Vector{\alpha}}}\right)^{\lambda^{\Vector{\beta}}_{\Vector{\alpha}}}}{= |b_{\Vector{\beta}}|}{\Vector{\beta} \in \nonSquares{p}}
	\myCons{X_{\Vector{\beta}, \Vector{\alpha}}}{\geq 0}{\Vector{\alpha} \in \monoSquares{p}, \Vector{\beta} \in \nonSquares{p}}.
\end{maxi}
Then we obtain our lower bound $\boundSONC$.

For simplicity, we restricted to the case, where every non-square occurs in exactly one circuit.
Therefore, the size of this covering is bounded by $|\cover|\in\mathcal O(t)$.
If some $\Vector{\beta}$ occurs in multiple circuits, the coefficient $b_{\Vector{\beta}}$ has to be distributed among them and the second set of constraints has to be adjusted accordingly.
In the original approach, we checked for every circuit, which other exponents of non-squares are included in its Newton polytope and added these circuit as well.
This extended the size of the covering to $|\cover| \in \mathcal O(t^2)$.
In either case, the size of \cref{problem:SONC} is polynomially bounded in the input size.

\subsection{Computing Minima}
\label{subsection:minima}

While the previous section describes an algorithm to compute lower bounds for multivariate polynomials via SONC, we do not have upper bounds for the minimum, so we do not have any guarantees for the quality of our bounds.
While every local minimum gives such a bound, we can use the SONC decomposition for a heuristic to find a good (local) minimiser.
Müller investigated this idea in further detail in \cite{Mueller:SONC_minima} and in this section, we present the main ideas of her work.
In short, given a SONC decomposition, we explicitly compute the minimiser of each circuit polynomial.
Then we take the barycentre of these minimisers and use it as starting point for some local minimisation method.

Müller's experimental results were already promising.
We re-implemented the approach and include it in our software POEM.
See \cref{section:experimental_results} for the experimental results.

Generalising work from Iliman and de Wolff \cite{Iliman:deWolff:Circuits}, Müller shows the following theorem.
For the reader's convenience, we also provide the proof here.
\begin{theorem}[{\cite[Theorem 2.5]{Mueller:SONC_minima}}]
	\label{theorem:circuit_minimiser}
	For a circuit polynomial $p$ with $\alpha(0) = 0$ and $b_{\Vector{\beta}}<0$, let $\Vector{s}$ be the vector satisfying the linear equation system
	\begin{align}
		\langle \Vector{s}, \Vector{\alpha}(j) - \Vector {\beta}\rangle &= \log\left(- \frac{\lambda_j}{b_{\Vector{\alpha}(j)}} \cdot b_{\Vector{\beta}}\right)
		&\text{for all $1\leq j\leq n$.}
		\label{equation:condition_min}
	\end{align}
	Then $e^{\Vector{s}}$ is the global minimiser of $p$.
\end{theorem}
\begin{proof}
	Condition \cref{equation:condition_min} implies $e^{\langle\Vector{s},\Vector{\alpha}(j)\rangle} = -\frac{\lambda_j}{b_{\Vector{\alpha}(j)}} b_{\Vector{\beta}} \cdot e^{\langle\Vector{s},\Vector{\beta}\rangle}$.
	Evaluating the shifted partial derivative at $e^{\Vector{s}}$ yields
	\begin{align*}
		\left(x_j\frac{\partial p}{\partial x_j}\right)\left(e^{\Vector{s}}\right)
		&= \sum_{k=1}^n b_{\Vector{\alpha}(k)} \Vector{\alpha}(k)_j e^{\langle \Vector{s}, \Vector{\alpha}(k)\rangle} + b_{\Vector{\beta}} \beta_j e^{\langle\Vector{s},\Vector{\beta}\rangle}
		= -\sum_{k=1}^n \lambda_k b_{\Vector{\beta}} \Vector{\alpha}(k)_j e^{\langle \Vector{s}, \Vector{\beta}\rangle} + b_{\Vector{\beta}} \beta_j e^{\langle\Vector{s},\Vector{\beta}\rangle} \\
		&= b_{\Vector{\beta}} e^{\langle \Vector{s}, \Vector{\beta}\rangle} \left(-\sum_{k=1}^n \lambda_k \Vector{\alpha}(k)_j + \beta_j\right)
		= 0
	\end{align*}
	Note that the 0-th summand vanished, since $\Vector{\alpha}(0)=\Vector{0}$ and the final sum vanishes, because
	\begin{align*}
		\Vector{\beta} = \sum_{k=0}^n \lambda_k \Vector{\alpha}(k) = \sum_{k=1}^n \lambda_k \Vector{\alpha}(k)
	\end{align*}
	Hence, $e^{\Vector{s}}$ is a local minimiser of $p$. 
	By \cite[Proposition 3.3]{Iliman:deWolff:Circuits}, it is the unique minimum in the positive orthant and since $b_{\Vector{\beta}}<0$, the global minimum is attained in the positive orthant.
\end{proof}
Since $\Vector{\beta}$ lies in the interior of the Newton polytope, the vectors $\Vector{\beta}, \Vector{\alpha}(1), \ldots, \Vector{\alpha}(n)$ span a simplex as well.
Hence, the vectors $\Vector{\alpha}(j)-\Vector{\beta}$ are linearly independent.
Therefore, \cref{equation:condition_min} has a unique solution, so it is justified to speak of \emph{the} solution $\Vector{s}$.
For $b_{\Vector{\beta}}>0$, the inner term either is a monomial square, or has an odd power.
If $b_{\Vector{\beta}}\Vector{x}^{\Vector{\beta}}$ is a monomial square, the minimiser trivially is $\Vector{0}$.
Otherwise, let $\beta_i$ be odd and put $\hat{p} = p(x_1,\ldots,-x_i,\ldots,x_n)$.
Then $\hat{p}$ satisfies the conditions of \cref{theorem:circuit_minimiser} and has the same infimum.

Now let $\relaxPoly{p} = p_1+\ldots+p_{\coverLength}$ be the SONC decomposition of the relaxation of $p$.
Let $\minimum_i = e^{\Vector{s}_i}$ be the respective minima of the $p_i$ according to \cref{theorem:circuit_minimiser}.
Then we use the barycentre $\overline{\minimum} := \frac{1}{\coverLength}\sum_{i=1}^{\coverLength} \minimum_i$ as starting point for a gradient method to find a local minimum of $\relaxPoly{p}$, which we denote $\minimum$.
The expectation is, that $\overline{\minimum}$ often lies sufficiently close to the global minimum of $\relaxPoly{p}$.
In these cases, this local minimum will also be the global minimum of $\relaxPoly{p}$.
If $p \neq \relaxPoly{p}$, we now call a gradient method on $p$ with $\minimum$ as starting point, to obtain our final result.
\newcommand{\gradientDescent}[1]{\operatorname{LocalMin}\left(#1\right)}
\begin{algorithm}
	The algorithm to compute (local) minima via SONC in polynomial time works as follows.
	For given accuracy $\varepsilon$, the running time is polynomial in the input size and $\frac{1}{\varepsilon}$.
	\begin{algorithmic}[1]
		\Require $p$ -- Polynomial
		\Ensure $\minimum$ -- Local minimum of $p$
		\Function{SONC-Min}{p}
			\State compute SONC-decomposition $\relaxPoly{p} = p_1+\ldots+p_{\coverLength} + \variableStyle{const}$
			\For{$i=1,\ldots,\coverLength$}
				\State set $\Vector{s}_i$ as solution of \cref{equation:condition_min} for $p_i$
				\State $\minimum_i \gets e^{\Vector{s}_i}$ \Comment{minimum of circuit-polynomial $p_i$}
			\EndFor
			\State $\minimum \gets \frac{1}{\coverLength}\sum_{i=1}^{\coverLength} \minimum_i$
			\State $\minimum \gets \gradientDescent{\relaxPoly{p}, \minimum}$ \Comment{local minimum of $\relaxPoly{p}$}
			\State \Return $\minimum \gets \gradientDescent{p, \minimum}$ \Comment{local minimum of $p$}
				\label{line:min_p}
		\EndFunction
	\end{algorithmic}
\label{algorithm:min_polynomial}
\end{algorithm}
\begin{proof}
	As observed in \cref{subsection:SONC_algo}, we have $\coverLength \in\mathcal O(t^2)$, 
	For each minimiser, we have to solve a linear equation system, which can be done in $\mathcal O(n^3)$.
	For the local minimum, we can use nonlinear gradient descent, which has quadratic convergence \cite{Fletcher:Reeves:conjugate_gradient}.
	Hence, we have an overall polynomial running time.
\end{proof}

\subsection{SAGE Polynomials}
\label{subsection:background_sage}

In \cite{Chandrasekaran:Shah:SAGE} Chandrasekaran and Shah introduce another certificate for nonnegativity, based on ``sums of arithmetic-geometric-mean exponentials'' (SAGE).
They also form a class of sparse polynomials, whose nonnegativity can also easily be verified.
Both computing this certificate and decomposing a polynomial as SAGE (if possible) can be done by a \struc{relative entropy programme}.
Like for SONC, the support of the certificate of nonnegativity is exactly the support of the input polynomial, which also makes this approach well-suited, to obtain lower bounds for sparse polynomials.
In fact, both approaches describe the same set of polynomials \cite{Murray:Chandrasekaran:Wierman:SONC_is_SAGE}.

A \struc{signomial} is an expression of the form
\begin{align*}
	p = \sum_{j=1}^t b_j \cdot \exp\left(\langle\Vector{\alpha}(j), \Vector{x}\rangle\right)
\end{align*}
with $b_j\in\mathds R$ and $\Vector{\alpha}(j)\in\mathds N^n$.
Via logarithmic transformation, signomials correspond to polynomials, whose domain is restricted to $\mathds R_+^n$.

An \struc{\emph{arithmetic-geometric-mean-exponential (AGE)}} is a nonnegative signomial with at most one negative coefficient.
The name comes from the fact that its nonnegativity can be verified via arithmetic-geometric-mean inequality.
The \emph{sums of AGE polynomials (SAGE)} form a convex cone.
Testing membership in this cone can be done by solving a \struc{relative entropy programme (REP)}, which is a type of convex optimisation problem.

The \struc{\emph{relative entropy}} function is defined for $\Vector{\lambda}, \Vector{v} \in \mathds R_+^t$ by $\struc{D(\Vector{\lambda},\Vector{v})} := \sum_{j=1}^t \lambda_j \log \frac{\lambda_j}{v_j}$. 
Furthermore, let $\Vector{v}_{\setminus i}\in\mathds R^{n-1}$ denote the vector derived from $\Vector{v}\in\mathds R^{n}$, where the entry at index $i$ was removed and for a matrix $\Matrix{X}$ let $\Vector{X}^{(j)}$
Then from \cite[Proposition~2.4]{Chandrasekaran:Shah:SAGE}, we have the following characterisation.
\begin{theorem}
	\label{theorem:SAGE_characterisation}
	A signomial $p=\sum_{j=1}^t b_j \exp \, (\langle \Vector{\alpha}(j), \Vector{x}\rangle)$ lies in SAGE if and only if there are $\Matrix{X}$, $\Matrix{\lambda} \in \mathds R^{t\times t}$ satisfying the following conditions:
	\begin{align}
		\label{problem:SAGE-feas}
		\begin{aligned}
			\sum_{i=1}^t \Vector{X}^{(i)} = \Vector{b} \,, \quad
			\sum_{j=1}^t \Vector{\alpha}(j) \Vector{\lambda}_j^{(i)} = \Vector{0} \,, \quad
			-\Vector{1} \cdot \Vector{\lambda}_{\setminus i}^{(i)} = \lambda_i^{(i)} \,, \\
			\Vector{X}_{\setminus i}^{(i)},	\Vector{\lambda}_{\setminus i}^{(i)} \geq \Vector{0} \,, \quad
			D\left(\Vector{\lambda}_{\setminus i}^{(i)}, e\Vector{X}_{\setminus i}^{(i)}\right) \leq X_i^{(i)} 
		 \,, \quad i=1,\dots,t \,.
		\end{aligned}
		\tag{SAGE-feas}
	\end{align}
\end{theorem}
One way to obtain lower bounds of a signomial $f$ is to solve the following REP:
\begin{align}
	\label{problem:SAGE}		
		\struc{p_{\text{SAGE}}} := \sup \{ b \in \mathds R: p - b \text{ is SAGE}\}.
	\tag{SAGE}
\end{align}
The constraints of~\cref{problem:SAGE} correspond to~\cref{problem:SAGE-feas}, after replacing $b_0$ by $b_0 - C$.

The second type of constraints has a size in $\mathcal O(n)$ but by restricting to the ambient space, we may assume $n\leq t$.
So the overall size of both the decision and the optimisation problem lies in $\mathcal O(t^2)$.
Most notably, it is independent of the degree $d$.
Recall, that we investigate sparse polynomials, which means $t \ll \binom{n+d}{d}$.

\subsection{Parametrised Complexity}
When solving a problem, one mainly is interested in \struc{efficient} algorithms, which usually means a running time \struc{polynomial in the input length}.
However, even a theoretically exponential time algorithm might be practically feasible, if the exponential part is sufficiently small.
These considerations have led to a whole hierarchy of complexity classes, but in this paper, we are only interested in the class of \struc{fixed-parameter tractable problems (\FPT{})}.
See \cite{Grohe:Logic:Graphs:Algorithms} for more details.

\begin{definition}
	\label{def:parametrised_problem}
	A \struc{\emph{parametrised problem}} is a pair $(P,\kappa)$ such that $P\subseteq\Sigma^*$ is a language and $\struc{\kappa}:\Sigma^*\to\mathds N$ is called the \struc{parameter}.
\end{definition}

\begin{definition}
	\label{def:FPT}
	The class \struc{\emph{\FPT{}}} is the class of all parametrised problems $(P,\kappa)$, where there exists a computable function $f:\mathds N\to\mathds N$ and a constant $c$, such that $x$ can be decided in time $\mathcal O(f(\kappa(x)) \cdot |x|^c)$.
\end{definition}
Note, that increasing the parameter here only affects a factor of the running time, but not the exponent $c$.
So for moderate values of the parameter, these problems can often be solved in practice.

In contrast, the often found description ``polynomial time for constant parameter'' describes the class $\struc{\XP}$.
More formally, it contains all problems such that there is a computable function $g:\mathds N\to\mathds N$ such that the problem can be solved in $\mathcal O\left(|x|^{g(\kappa(x))}\right)$.
We have strict containment $\FPT\subset\XP$, see e.g. \cite[Corollary 2.26]{Flum:Grohe:Para}.

\section{Branch and Bound}
\label{section:branch_and_bound}

In our previous paper \cite{Seidler:deWolff:POEM}, we described a method to obtain lower bounds for polynomials, as also given in \cref{subsection:SONC_algo}.
As initial step, we relaxed the polynomial by giving every possibly negative term a negative sign and then restricting ourselves to the positive orthant.
However, this is overly pessimistic, as can be seen in the following example.

\begin{example}
	Let $p=x^4+x^3-x+1$, which has minimum $\approx 0.682$.
	This polynomial is relaxed to $\relaxPoly{p} = x^4-x^3-x+1$, which has minimum $0$.
	\label{example:relaxation}
	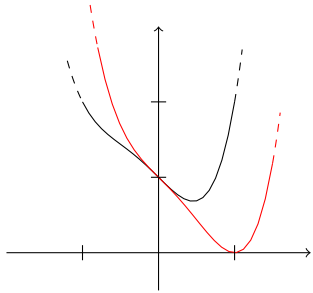
\begin{figure}[!htbp]
		\centering
		\begin{tikzpicture}
			\draw[->] (0,-0.5) -- (0,3);
			\draw[->] (-2,0) -- (2,0);
			\draw plot[domain=-1:1] ({\x},{\x*\x*\x*\x + \x*\x*\x - \x + 1});
			\draw[dashed] plot[domain=-1.2:-1] ({\x},{\x*\x*\x*\x + \x*\x*\x - \x + 1});
			\draw[dashed] plot[domain=1:1.1] ({\x},{\x*\x*\x*\x + \x*\x*\x - \x + 1});
			\draw[red] plot[domain=-0.8:1.5] ({\x},{\x*\x*\x*\x - \x*\x*\x - \x + 1});
			\draw[red, dashed] plot[domain=-0.9:-0.8] ({\x},{\x*\x*\x*\x - \x*\x*\x - \x + 1});
			\draw[red, dashed] plot[domain=1.5:1.6] ({\x},{\x*\x*\x*\x - \x*\x*\x - \x + 1});
			\draw[-] (1,-0.1) -- (1,0.1);
			\draw[-] (-1,-0.1) -- (-1,0.1);
			\draw[-] (-0.1,1) -- (0.1,1);
			\draw[-] (-0.1,2) -- (0.1,2);
		\end{tikzpicture}
		\caption{Graphs of $p$ (black) and $\relaxPoly{p}$ (red)}
		\label{figure:example_relaxation}
	\end{figure}
\end{example}

To overcome this problem, we propose a branch-and-bound algorithm, where we branch over the signs of the variables.
By fixing a sign for a variable, some terms with an odd power are then known to be positive, so they can be regarded as monomial squares.
Hence, we do not have to cancel out their negative weight, but in addition gain new positive weights to cancel out the remaining negative terms.

To denote our restrictions on the signs of the variables, we introduce sign cones.
\begin{definition}
	\label{definitoin:sign_cone}
	Let $\Vector{s}\in\{-1,0,1\}^n$.
	We call $\Vector{s}$ a \struc{sign vector}, where $-1,0,1$ represents negative/unknown/positive sign, respectively, and define the corresponding \struc{sign cone} as
	\begin{align*}
		\struc{\signCone{s}} := \left\{\Vector{x}\in\mathds R^n: x_i\cdot s_i \geq 0 \text{ for all } i=1,\ldots,n\right\}.
	\end{align*}
	For some sign vector $\Vector{s}$, the \struc{positive points} are given by
	\begin{align*}
		\posPoints{p}{\Vector{s}} = \left\{\Vector{\alpha} \in \support{p} : \operatorname{sgn}(\Vector{\alpha}) \cdot \prod_{i=1}^n s_i^{\Vector{\alpha}_i \bmod 2} = 1\right\}
	\end{align*}
	with the convention $0^0=1$.
	The \struc{negative points} are $\negPoints{p}{s} := \support{p} \setminus\posPoints{p}{s}$.
	The corresponding \struc{positive} and \struc{negative terms} are the terms $b_{\Vector{\alpha}}\Vector{x}^{\Vector{\alpha}}$ for $\Vector{\alpha}\in \posPoints{p}{s}$ and $\Vector{\alpha}\in \negPoints{p}{s}$, respectively.
	The whole polynomial restricted to the domain $\signCone{s}$ we denote by $\struc{p_{\Vector{s}}}$.
\end{definition}
In this notation $p = p_{\Vector{0}}$.
Clearly, $\monoSquares{p} \subseteq \posPoints{p}{s}$ for any sign vector $\Vector{s}$.
So we obtain the new relaxation
\begin{align*}
	\struc{\relaxPoly{p}_{\Vector{s}}} = \sum_{\Vector{\alpha} \in \posPoints{p}{s}} b_{\Vector{\alpha}}\Vector{x}^{\Vector{\alpha}} - \sum_{\Vector{\beta} \in \negPoints{p}{s}} \left|b_{\Vector{\beta}}\right| x^{\Vector{\beta}}.
\end{align*}
For $\monoSquares{p} \neq \posPoints{p}{s}$, this is an improvement of the original relaxation.
Note, that by fixing more signs, the set of positive terms may only grow.

Our branch-and-bound algorithm creates a binary search tree, where each node has a sign vector and a flag, whether it is active or not.
Furthermore, we store the best known lower bound and the lowest found function value over the corresponding sign cone.
For simplicity, we identify the nodes with their sign vectors and denote the lower bounds as $\lowerBound{p_{\Vector{s}}}$.

Initially, the tree consists only of the root node, which is active and has sign vector $\Vector{s} = \Vector{0}$.
So it corresponds to $\signCone{s}=\mathds R^n$ as the domain and we compute bounds as in \Cref{subsection:SONC_algo}.
In each iteration, we then pick some active node, which becomes inactive.
If it satisfies any bounding criterion, we continue with the next iteration.
Otherwise, we determine some index $i$ with undetermined sign $s_i=0$ and create two new child nodes for $\Vector{s}$, where we update $\Vector{s}$ with $s_i=\pm 1$.
We compute lower bounds and minimisers for both nodes and mark them as active.
Then we continue with the next iteration.

\begin{algorithm}
	We have the following branch-and-bound blueprint.
	\begin{algorithmic}[1]
		\Require $p$ -- Polynomial
		\Ensure \lowerBound{p} -- Lower bound of $p$
		\Function{Branch}{p}
			\State run SONC, SAGE on $p$
			\State $\variableStyle{active} := \{p\}$
			\While{$\variableStyle{active} \neq \emptyset$}
				\State pick some $p_{\Vector{s}} \in \variableStyle{active}$
					\label{line:pick_active}
				\State $\variableStyle{active} = \variableStyle{active} \setminus\{p_{\Vector{s}}\}$
				\If{$p_{\Vector{s}}$ does not satisfy bounding criterion} 
						\label{line:bounding}
					\State pick index $i$ with $s_i=0$
						\label{line:pick_index}
					\State set $\Vector{s}^+$ as $\Vector{s}$ with $s_i=1$ and $\Vector{s}^-$ as $\Vector{s}$ with $s_i=-1$
					\State compute SONC/SAGE bounds for $p_{\Vector{s}^+}$ and $p_{\Vector{s}^-}$
					\State $\variableStyle{active} = \variableStyle{active} \cup \{p_{\Vector{s}^+},p_{\Vector{s}^-}\}$
				\EndIf
			\EndWhile
		\EndFunction
	\end{algorithmic}
	\label{algorithm:branch_and_bound_blueprint}
\end{algorithm}
Termination follows immediately, since there are at most $2^{n+1}-1$ polynomials involved, and once a polynomial is removed from $\variableStyle{active}$, it is never inserted again.

To investigate the algorithm, let \struc{$\subtree{\Vector{s}}$} denote the current subtree, with root $\Vector{s}$.
This means, we may regard the search tree at any \emph{intermediate state}, and the root of this subtree need not correspond to all of $\mathds R^n$.
First we observe the following.

\begin{remark}
	\label{remark:better_than_parent}
	Each node has a bound at least as good as its parent.
	From the mathematical side, this is clear since we restrict the domain.
	But also from the algorithmic side, the certificate of the parent is also a certificate for the node itself.
\end{remark}

With this observation, we can see how the search tree yields a global bound for a polynomial $p$.
\begin{lemma}
	A lower bound for the polynomial $p_{\Vector{s}}$ is given by
	\begin{align*}
		\inf p &\geq \min\left\{\max\left\{\lowerBound{v}: v\in P\right\}: P\text{ is maximal path in $\subtree{\Vector{s}}$}\right\} 
		= \min\left\{\lowerBound{l} : l \text{ is leaf in }\subtree{\Vector{s}}\right\}
	\end{align*}
	for any intermediate state of the search tree $\subtree{\Vector{s}}$.
	In particular, we get a lower bound for $p$ from the tree $\subtree{\Vector{0}}$.
	\label{lemma:bound_formula}
\end{lemma}
\begin{proof}
	Since any node has lower bound at least as good as its parent, the maximum of any path is attained at its endpoint, which is a leaf of $\subtree{\Vector{s}}$.
	The sign cones, represented by the leaves, partition the whole space $\signCone{s}$.
	Put $M:=\min\left\{\lowerBound{l} : l \text{ is leaf in }\subtree{\Vector{s}}\right\}$.
	Let $\Vector{s}_1,\ldots,\Vector{s}_r$ be the leaves of $\subtree{\Vector{s}}$.
	Then, in each of these sign cones, we have $\inf p_{\Vector{s}}\geq \lowerBound{p_{\Vector{s}_i}}\geq M$.
	Hence, globally we have $\inf p_{\Vector{s}}\geq M$.
\end{proof}

However, the method in \cref{algorithm:branch_and_bound_blueprint} so far only presents a blueprint.
The following steps still have to be made more precise.
\begin{itemize}
	\item Which node $\Vector{s}$ do we pick in \cref{line:pick_active}?
	\item What are our bounding criteria in \cref{line:bounding}?
	\item Which index $i$ do we choose in \cref{line:pick_index}?
\end{itemize}
We address these issues in the following subsections.

\subsection{Bounding Criteria}
\label{subsection:bounding}

A crucial part in a branch-and-bound algorithm is to have efficient bounds.
So we need some easily checkable criteria, which allow us to cut off a branch of the search tree.

\begin{definition}
	\label{definition:cut_criteria}
	Let $\Vector{c}$ be our current node.
	We define the following criteria for cutting off branches.
	\begin{description}
		\item[\namedlabel{item:min}{Min}]
			If we have found some argument $\Vector{x}$ such that $p(\Vector{x})\leq \lowerBound{p_{\Vector{c}}}$, then we cut off the branch at $\Vector{c}$.
		\item[\namedlabel{item:leaf}{Leaf}]
			If there is some leaf $\Vector{s} \in \{-1,1\}^n$, i.e. leaf $\Vector{s}$ lies at depth $n$, with $\lowerBound{p_{\Vector{s}}} \leq \lowerBound{p_{\Vector{c}}}$, then we cut off the branch at $\Vector{c}$.
			\qedhere
	\end{description}
\end{definition}

\begin{lemma}
	The above cut criteria are correct.
	\label{lemma:cut_criteria_correct}
\end{lemma}
\begin{proof}
	Since both criteria are independent, we show their correctness separately.
	\begin{description}
		\item[\ref{item:min}]
			According to \cref{lemma:bound_formula}, our lower bound is
			\begin{align*}
				M:=\min\left\{\lowerBound{p_{\Vector{s}}} : \Vector{s} \text{ is leaf in } T\right\}.
			\end{align*}
			As observed before, in any state, the sign cones cover $\mathds R^n$, so let $\Vector{x} \in \signCone{s}$ for some $\Vector{s}$.
			Then, of course, $\lowerBound{p_{\Vector{s}}} \leq p(\Vector{x})$.
			Hence, branching $\Vector{c}$ any further, would only increase $\lowerBound{p_{\Vector{c}}}$ but not affect the minimum.
			So we can cut off the branch at $\Vector{c}$.			
		\item[\ref{item:leaf}]
			Since $\Vector{s} \in \{-1,1\}^n$, we cannot branch it any further, so its bound will not improve.
			Therefore, $M\leq \lowerBound{p_{\Vector{s}}}$.
			Again, increasing $\lowerBound{p_{\Vector{c}}}$ will then not affect the minimum, so can cut off the branch at $\Vector{c}$.
		\qedhere
	\end{description}
\end{proof}

Both criteria can easily be checked.
We separately store the lowest value, we have found, and the lowest value of some leaf $\lowerBound{p_{\Vector{s}}}$ for $\Vector{s}\in\{-1,1\}^n$.
(As long as we have not processed any such node, this value is $\infty$.)
Then both \ref{item:min} and \ref{item:leaf} are single comparisons, which take $\mathcal O(1)$ time.

\subsection{Choice of Branching Node}
\label{subsection:bnb_branch_choice}

Another problem to be addressed is the choice of the node, on which to branch in \cref{line:pick_active}.

With regard to the quality of the solution, the best choice is to choose the node with the smallest lower bound.
As we have seen in \cref{lemma:bound_formula}, the final bound is the smallest bound of any leaf.
Hence, if we do not improve this worst bound, the final bound will not improve.
The main disadvantage is, that the number of active leaves can grow exponentially, so this requires space exponential in $n$.

If memory is an issue, then the tree should be traversed in a depth-first-search.
With this strategy, we ensure $|\variableStyle{active}|\leq n+1$, so the computation runs in polynomial space.
The significant disadvantage is the higher running time, because we compute more nodes than with the previous strategy.

\subsection{Practical Improvements}
\label{subsection:bnb_practical_improvements}

If we choose the node with the worst bound for further branching, then the criterion \ref{item:min} never applies.
However, for numerical computations, we use a relaxed version.
Let $p(\Vector{m})$ be the lowest function value we found so far and let $\varepsilon$ be some given accuracy.
If for the current node $\Vector{c}$ we have $\lowerBound{p_{\Vector{c}}} \geq p(\Vector{m}) + \varepsilon$, we stop the whole computation, because we already have solved the problem up to accuracy $\varepsilon$.
To use this criterion, we integrate the computation of local minima via \SONCmin{} into the branch-and-bound method.
The global minimum is also more likely to be found in the sign cone with the worst lower bound.
So, whenever we compute bounds for a sign cone, we also search for a  local minimiser via \cref{algorithm:min_polynomial}.
However, the minima we compute for the circuit polynomials always lie in the positive orthant.
Hence, our starting point $\relaxPoly{\minimum}$ for the minimisation lies in the positive orthant as well.
To comply with our orthant restriction, we flip some signs of our starting point and define
\begin{align*}
	\relaxPoly{\minimum}' = \begin{cases}
		-\relaxPoly{m}_i &: s_i = -1 \\
		\relaxPoly{m}_i &: \text{else}
	\end{cases}.
\end{align*}
as use $\relaxPoly{\minimum}'$ as our starting point in \cref{line:min_p}, where we run gradient descent for $p$.

Furthermore, we adjust \ref{item:leaf} as follows.
Once we reach a node $\Vector{s}\in\{-1,1\}^n$, i.e. all signs are known we stop the whole algorithm.
We cannot improve this node by further branching, since by our choice of $\Vector{s}$ and \cref{lemma:bound_formula}, we already have $\lowerBound{p} = \lowerBound{p_{\Vector{s}}}$.
Hence, we cannot improve the bound of $p$ by our approach any more.

All combined, this yields the following algorithm.
\begin{algorithm}
	We have the following branch-and-bound algorithm, whose running time is fixed parameter tractable in the number of variables $n$.
	\begin{algorithmic}[1]
		\Require $p$ -- Polynomial
		\Ensure \lowerBound{p} -- Lower bound of $p$
		\Function{\algoBnB{}}{p}
			\State run SONC, SAGE on $p$
			\State $\variableStyle{min} \gets \operatorname{SONC-Min}(p)$
			\State $\variableStyle{active} \gets \{p\}$
			\While{$\variableStyle{active} \neq \emptyset$}
			\State $\Vector{s} \gets \operatorname{argmin} \left\{ \lowerBound{p_{\Vector{s}}} : \Vector{s} \in \{-1,0,1\}^n, p_{\Vector{s}} \in \variableStyle{active} \right\}$
					\label{line:worst_bound}
					\Comment{cone with worst bound}
				\State $\variableStyle{active} \gets \variableStyle{active} \setminus\{p_{\Vector{s}}\}$
				\If{$p_{\Vector{s}}$ satisfies \ref{item:min} or \ref{item:leaf}} 
					\Comment{see \Cref{subsection:bounding}}
				\State \Return $\lowerBound{p_{\Vector{0}}}$
				\EndIf
				\State compute SONC/SAGE for $p_{\Vector{s}^+}$ and $p_{\Vector{s}^-}$
				\State propagate new bound upwards
				\If{$\operatorname{SONC-Min}(p_{\Vector{s}^+}) < \variableStyle{min}$ or $\operatorname{SONC-Min}(p_{\Vector{s}^-}) < \variableStyle{min}$}
					\State update $\variableStyle{min}$
				\EndIf
				\State $\variableStyle{active} \gets \variableStyle{active}\cup\{p_{\Vector{s}^+}, p_{\Vector{s}^-}\}$
			\EndWhile
		\EndFunction
	\end{algorithmic}
	\label{algorithm:branch_and_bound}
\end{algorithm}
\begin{proof}
	Our search tree is a binary tree of height at most $n$, so it has at most $2^{2n+1}-1$ nodes, which means at most $2^{n+1}-1$ different polynomials are involved.
	Furthermore, in each state, exactly the leaves are in $\variableStyle{active}$ and we never remove nodes from the tree.
	Therefore, each polynomial is chosen at most once in the loop in \cref{line:worst_bound}.
\end{proof}

As a final variation, it turned out that SAGE takes significantly longer than SONC, but for most sign cones the bound computed via SONC suffices.
So, initially we only compute a lower bound via SONC for each $p_{\Vector{s}}$.
If some node is chosen for the first time, we then compute a lower bound via SAGE and the node remains active.
Only if this node is chosen a second time, it becomes inactive and we branch into the two sub-cones.

\section{Minimal Orthants}
\label{section:minimal_orthants}

As alternative to the branch-and-bound algorithm, we can find a sufficient subset of the leaves and directly compute lower bounds for these polynomials.

As soon as we are given a concrete orthant,~i.e. we know the sign of every variable, we can compute the \struc{effective sign} of each term,~i.e. we know whether it is positive or negative.
To keep consistent with our previous notation for the relaxation, we also denote this polynomial as $\relax{p} = \left(A(p), \relax{b(p)}\right)$.
Now we define a partial order $b_1\leq b_2$ on the effective coefficient vectors as elementwise $\leq$.
This lifts to a partial order on the polynomials
\begin{align*}
	\polynomial{A}{b_1} \leq \polynomial{A}{b_2} :\Leftrightarrow b_1\leq b_2.
\end{align*}
Going over all orthants yields $2^n$ polynomials.
But the crucial observation is that we only need to compute bounds for the \emph{minimal} polynomials.
\begin{example}
	\label{example:minimal_orthants}
	Consider the following polynomial with 3 variables.
	\begin{align*}
		p&=2.723 + 3.932\cdot x_2^{8} + 6.054\cdot x_1^{2} + 1.963\cdot x_1^{4} x_2^{2} - 1.204\cdot x_0^{1} x_1^{1} x_2^{3} + 1.462\cdot x_0^{1} x_1^{2} x_2^{1}\\
		&\quad+ 1.766\cdot x_0^{1} x_1^{2} x_2^{2} + 0.841\cdot x_0^{1} x_1^{2} x_2^{4} - 0.329\cdot x_0^{2} x_1^{1} x_2^{2} + 7.57\cdot x_0^{2} x_1^{2} x_2^{4} + 2.428\cdot x_0^{4} x_2^{2}
	\end{align*}
	Then the minimal orthants are given by the signs $(-,+,+)$, $(-,+,-)$ and $(-,-,+)$.
	So instead of solving $8=2^3$ instances, we only have to solve the three instances where we restrict $p$ to each of the above orthants.
\end{example}

\subsection{Computing Minimal Orthants}
\label{subsection:computing_minimal_orthants}

For convenience, we define the indicator function for strictly negative terms
\begin{align*}
	\struc{\negative}(x) = \begin{cases}1&:x<0\\0&:x\geq0\end{cases}.
\end{align*}
If called on a vector, the function is applied elementwise.

\begin{algorithm}
	Computing the orthants with minimal coefficient vector is fixed parameter tractable in $n$, via the following algorithm.
	\begin{algorithmic}
		\Require $p$ -- Polynomial
		\Ensure $\minList$ -- Set of orthants, where coefficients have minimal effective sign
		\Function{MinimalOrthants}{p}
			\State $\minList = \emptyset$
			\For{$\variableStyle{sign} \in \{0,1\}^n$} \Comment{fork over all orthants}
				\State $v = \left(\variableStyle{sign} \cdot A + \negative(b)\right) \bmod 2$
				\For{$(e,s) \in \minList$}
					\If{$e \leq v$}
						\State continue with next $\variableStyle{sign}$ \Comment{$v$ is not minimal}
					\EndIf
					\If{$v < e$}
						\State $\minList.remove(e,s)$ \Comment{$e$ is not minimal}
					\EndIf
				\EndFor
				\State $\minList.add(v,\variableStyle{sign})$ \Comment{$v$ is minimal, if we reach the end of the for-loop}
			\EndFor
		\EndFunction
	\end{algorithmic}
	\label{algorithm:minimal_orthants}
\end{algorithm}
\begin{proof}
	Let $t'$ be the number of non-squares.
	The length of $\minList$ is bounded by both $2^n$ and the length of the maximal antichain $\binom{t'}{t'/2}$.
	Furthermore, each comparison runs in $\mathcal O(t')$.
	So the overall running time is 
	\begin{align*}
		\mathcal O\left(t'\cdot 2^n\cdot \min\left(2^n, \binom{t'}{t'/2}\right)\right) 
		\subseteq \mathcal O\left(t\cdot 4^n\right)
	\end{align*}
	which is fixed parameter tractable in $n$.
\end{proof}
In particular, the proof shows that this approach is useful for polynomials with few non-squares.


Experiments show that for $n=10$ variables and $t'=100$ this can be done in about 2 seconds.
We consider the problem of determining the minimal orthants practically feasible for values $n\leq 15$.

Then, we create polynomials
\begin{align*}
	p_v = \polynomial{A}{v} \quad\text{ for all }(v,\variableStyle{sign})\in \minList
\end{align*}
and optimise each over the positive orthant.
Hence, the running time significantly depends on the number of minimal orthants.
For polynomials with many monomial non-squares, we usually have $|\minList| = 2^n$, but for instances with few monomial non-squares, we significantly reduce the running time by restricting ourselves to the minimal orthants.
For the final lower bound, we then get
\begin{align*}
	\lowerBound{p} = \min\left\{\lowerBound{p_v} : (v,\variableStyle{sign}) \in \minList\right\}.
\end{align*}
The advantage of this approach, compared to the search tree, is its easy parallelisation.
The major disadvantage is that a numerical failure in a single polynomial $p_v$ already causes the trivial bound $\lowerBound{p} = -\infty$.
We discuss the quality of the results and the frequency of this problem in \cref{section:experimental_results}.

\subsection{Reducing the Search Tree}
\label{subsection:reducing_tree}

The idea of this section also gives rise to a variant of the branch-and-bound approach from \cref{section:branch_and_bound}.
First we compute the minimal orthants $\minList$ as described in \cref{subsection:computing_minimal_orthants}.
Then we create a tree whose leaves are the elements of $\minList$ and we branch the signs of the variables $x_1,\ldots,x_n$ in that order.
Whenever we compute a node of the tree, that only has a single child, we further descend down the tree, until we arrive at a node with two children, or a leaf.
Otherwise we apply the same algorithm as in \cref{section:branch_and_bound}, including the criteria for cutting off a branch.
We denote this algorithm by \algoBnB-sparse.

\section{Experimental Results}
\label{section:experimental_results}

We start by discussing the running time and the results of the algorithms presented above on a few selected examples.
Afterwards, we describe how the algorithms behaved on a large set of test cases.

\subsection{Experimental Setup}
\label{subsection:experimental:setup}

We give an overview about the experimental setup.

\textbf{Software} 
	The entire experiment was steered by our \textsc{Python} 3.7 based software \textsc{POEM} 0.3.0.0(a) (Effective Methods in Polynomial Optimisation), \cite{poem:software}, which we develop since July 2017. 
	\textsc{POEM} is open source, under GNU public license, and available at:
	\begin{center}
		\url{\POEMhomepage}
	\end{center}
	For our experiment, \textsc{POEM} calls a range of further software and solvers for computing the certificates.
	For the numerical solutions of SONC and SAGE, we use \textsc{CVXPY} 1.0.28 \cite{cvxpy}, to create the convex optimisation problems.
	To solve the problems, we use \ecos{} 2.0.7 \cite{ecos}, \mosek{} 9 \cite{mosek} and \cvxopt{} 1.2.2 \cite{Andersen:Dahl:Vandenberghe:CVXOPT}.

	As heuristic to find local minima of polynomials we use Müller's approach as decribed in \cref{subsection:minima}.
	In addition, we call local minimisation methods from random starting points and differential evolution from \textsc{SciPy} 1.4.1. \cite{scipy:software}.

\textbf{Investigated Data} 
	The experiment was carried out on a database containing \instancesBnB{} randomly generated polynomials.
	The possible numbers of variables are $n=\variables$; the degree takes various (even) values $4\leq d\leq 60$ and the number of terms can be $t=\Myterms$.

	We created the examples using \textsc{POEM}, and they are available in full at the homepage cited above.
	The instances investigated here, are a subset of those from \cite{Seidler:deWolff:POEM}.
	In that paper, we also describe their creation in more detail.
	The overall running time for all our instances was \totalTimeBnB{}~hours.

\textbf{Hardware and System} 
	We used an \verb+Intel Core i7-8550U+ CPU with 1.8 GHz, 4 cores, 8 threads and 16 GB of RAM under Ubuntu 18.04 for our computations.

\textbf{Stopping Criteria}
	For the accuracy of the solver and the precision of the rounding in \textsc{Python} we used a tolerance of $\varepsilon=2^{-23}$.
	We restrict ourselves to $n\leq 8$, to keep the potential factor $2^n$ for the running time in a reasonable range.
	Furthermore, when running SAGE, we noticed a significant increase in both run time, memory consumption and occurrences of numerical problems for $t\geq 100$. 
	Therefore, we restrict ourselves to $t\leq 50$ (which was the next lower number of terms in our example set).

\subsection{Selected Examples}

\begin{example}
	We consider a polynomial with $n=4$ variables, degree $d=16$ and $t=50$ terms.
	Here we particularly see, how the branch and bound approach significantly improved the bound.
	The lowest value we found, is $\xmin \approx 19.203$.
	We present lower bounds we obtained in \cref{table:example_10796}.

	The best bound was found by \algoBnB.
	The approach by \algoFork{} failed, because for at least one of the orthants, both SONC and SAGE encountered numerical issues.
	Since the overall bound is given by the worst bound on any of the orthants, we only obtain the trivial bound $-\infty$.

	Next, we observe, that the sparse version of \algoBnB{} here actually takes \emph{longer} than the standard version.
	Both methods compute 23 out of 31 possible nodes of the search tree.
	So the sparse method does not have any advantage.

	Furthermore, it computes a \emph{worse} bound.
	The reason for the latter is, that we branch the variables in a different order.
	At some point, SAGE runs into numerical problems, and these issues arise at different nodes in the search tree.
	The remaining bounds are then computed with the weaker, but more stable SONC method.
	Therefore, the two versions of \algoBnB{} can have different results.
	\begin{table}
		\begin{tabular}{ccccl}
			lower bound       & difference &  time   & strategy & options \\
			\hline
			11.992  	& 7.211      &   0.17  & SONC      \\
			13.693  	& 5.510      &   0.19  & SONC      & alternative covering\\
			14.458  	& 4.745      &   3.09  & SAGE      \\
			18.769  	& 0.434      &  53.98  & \algoBnB  \\
			18.284  	& 0.918      &  56.43  & \algoBnB, & sparse  \\
			$-\infty$ & $\infty$   &   2.26  & \algoFork, & SONC only      \\
			$-\infty$ & $\infty$   &   9.97  & \algoFork, & SONC and SAGE      \\
		\end{tabular}
		\caption{Comparison of time and quality of the lower bounds obtained by different approaches.
		Branch and bound yields the best bound.
		The forking algorithm numerically failed on some orthant, so its bound is $-\infty$.}
		\label{table:example_10796}
	\end{table}
	\label{example:10796}
\end{example}

\begin{example}
	The next example is a polynomial with $n=4$ variables, degree $d=10$ and $t=30$ terms.
	Here, \algoBnB-sparse computes the optimal bound $\xmin \approx 4.08$, while all of our other methods have an optimality gap of at least $0.27$.
	Together with \cref{example:10796}, it shows that \algoBnB{} and \algoBnB-sparse are in general incomparable in terms of their results.
	\begin{table}
		\begin{tabular}{ccccl}
			lower bound & difference &  time & strategy & options \\
			\hline
			-58.80      & 62.88      &  0.12 & SONC\\
			-51.84      & 55.92      &  0.20 & SONC& alternative covering\\
			-25.65      & 29.73      &  1.46 & SAGE\\
			  3.80      & 0.277      & 31.98 & \algoBnB \\
			  4.08      & 0          & 35.84 & \algoBnB , & sparse  \\
			  0.40      & 3.681      &  1.39 & \algoFork, & SONC only      \\        
			  3.80      & 0.277      &  4.29 & \algoFork, & SONC and SAGE      \\    
		\end{tabular}
		\caption{Comparison of time and quality of the lower bounds obtained by different approaches.
		Here, the sparse version of \algoBnB{} computed the best bound but also took the longest tie.
		All of our other methods computed worse bounds.}
		\label{table:example_10100}
	\end{table}
	\label{example:10100}
\end{example}


\subsection{Evaluation of the Experiment}

In this section, we summarise our findings from running our experiment on \instancesBnB{} instances.

\begin{description}
	\item[The bound of \algoBnB{} is at least as good as the bound of \algoFork{}.]
		\cref{figure:bar_plot_diff_traverse_fork} shows the difference between the lower bounds obtained by \algoBnB{} and \algoFork, where a positive value means that \algoBnB{} gave a better bound.
		In the majority of cases the difference is numerically zero, but in some cases, \algoBnB{} performs significantly better.
		In no case the differences goes below $-10^{-5}$.

		\begin{figure}
			
	\input{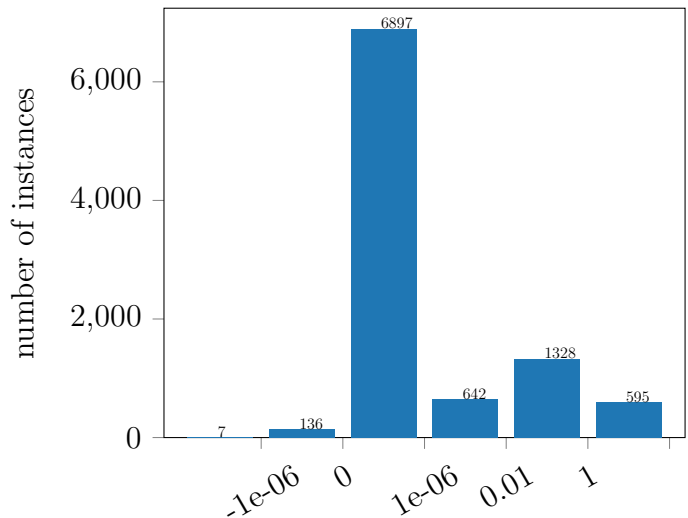}

			\caption{Difference between the bounds of \algoBnB{} and \algoFork{}.
			Practically, the result of \algoBnB{} is \emph{always} as least as good as the one of \algoFork{}.
			In the majority of our examples the difference is numerically zero.}
			\label{figure:bar_plot_diff_traverse_fork}
		\end{figure}
	\item[Sparse \algoBnB{} is slightly faster than standard \algoBnB.]
		The quotients of the running time for standard \algoBnB{} divided by the time of sparse \algoBnB{} range from \traverseQuotMin{} to \traverseQuotMax{} with a geometric mean of \traverseQuotGM{}.
		So, on average, the standard version takes about \traverseQuotPercent\% longer.
		The distribution of these quotients is shown with more detail in \cref{figure:bar_plot_traverse_quot}.
		In particular, in the majority of our cases, the run times differ by a factor of at most 2.
		\begin{figure}
			
	\input{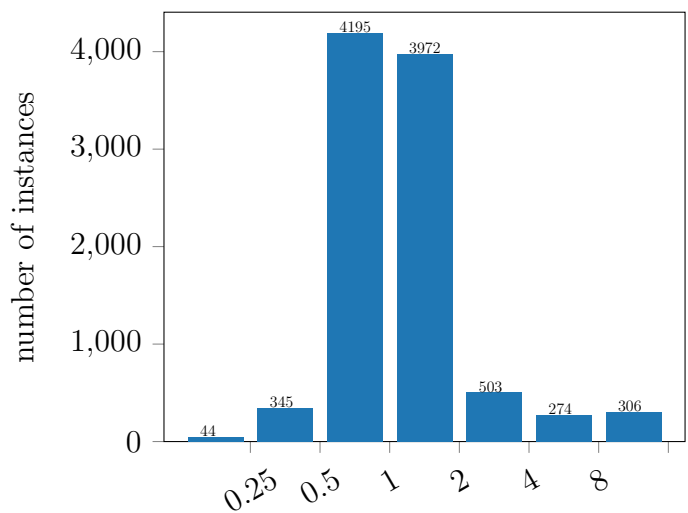}

			\caption{Quotient of runtimes of standard \algoBnB{} divided by sparse \algoBnB{}.
			In most cases the times differ by a factor of at most 2, but overall the sparse version has a slight advantage.}
			\label{figure:bar_plot_traverse_quot}
		\end{figure}
	\item[Failure of \algoFork{} is rare.]
		As seen in \cref{example:10796}, \algoFork{} may return the trivial bound $-\infty$.
		However, this only rarely happends.
		Among our test cases, there are only \boundForkFail{} instances, where \algoFork{} fails, but (at least one variant of) \algoBnB{} finds a lower bound.

	\item[Run time of \algoBnB{} and \algoFork, depending on $n,t$]
		As expected, the running time increases with both $n$ and $t$.
		In \cref{table:traverse_time} and \cref{table:fork_time}, we present the running times of \algoBnB{} and \algoFork{}, depending on $n$ and $t$.
		Since we observed in \cref{subsection:background_sage} and \cite{Seidler:deWolff:POEM}, that the running time is independent of the degree, we average over the degree as well.
		We can see, that the growth in $n$ is slower than $\Theta(2^n)$.
		Also note, that we have at least $n+1$ monomial squares, so for few terms the ratio of non-squares decreases with growing $n$.
		Thus, in these cases, the running time even \emph{decreases} with growing $n$.
		\newcommand{\fewInstances}[1]{#1$^*$}
		\begin{table}
			\centering
			\tableTraverseTime
			\caption{Average runtime of standard \algoBnB{}, depending on $n$ and $t$.
			A \fewInstances{} marks parameters, where we have less than 10 instances.
			}
			\label{table:traverse_time}
		\end{table}
		\begin{table}
			\centering
			\tableForkTime
			\caption{Average runtime of \algoFork{} with SAGE, depending on $n$ and $t$.
			A \fewInstances{} marks parameters, where we have less than 10 instances.
			}
			\label{table:fork_time}
		\end{table}

	\item[Optimality Gap]
		In \cref{figure:diff_value_bound}, we can see the distribution of the optimality gap among our instances, for how many instances the gap lies in the given interval.
		For the left bars (blue), we combined all of our new methods and took their best bound.
		The right bars (orange) show the distribution of the optimality gap, when just using SAGE.
		We computed the (local) minima via \algoBnB{} along with the lower bounds.

		For \smallGap{} instances, about \smallGapPercent{}\%, our methods yield a gap of at most $10^{-6}$, which we consider numerically zero.
		Furthermore, we see a clear improvement compared to using only a single call of SAGE.

		\begin{figure}
			
	\input{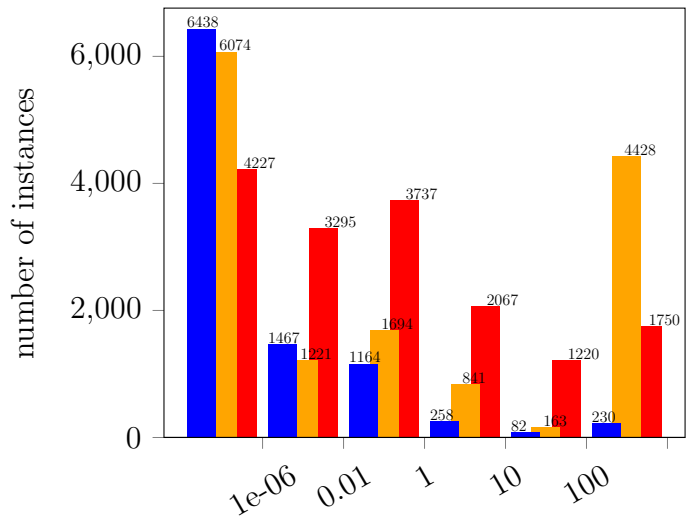}

			\caption{Distribution of the optimality gap;
			Left, in blue, we show the optimality gap for the combined methods standard-\algoBnB{}, sparse-\algoBnB{} and \algoFork{} with SAGE, taking their best bound.
			The right bars, in orange, show the optimality gap if we just use SAGE to compute a global lower bound.
			}
			\label{figure:diff_value_bound}
		\end{figure}
\end{description}

\section{Resume and Outlook}
\label{section:conclusion}

Our paper consists of two main contributions.
The first contribution is a branch-and-bound framework, where we branch over the signs of the variables.
Thus, additional terms can be identified as positive, which improves the lower bounds obtained via SONC and SAGE.
Second, we give an alternative to the branch-and-bound approach.
We identify the minimal orthants with respect to their sets of positive terms.
In these orthants we compute lower bounds for the given polynomial and the worst of these bounds is a global lower bound.
The two algorithms are fixed parameter tractable when parametrised by the number of variables.

We ran these methods on a larger number of test cases and draw the following conclusions.
\begin{enumerate}
	\item To obtain the best result, the method of choice is \algoBnB{}.
		However, between standard-\algoBnB{} and sparse-\algoBnB{} there is no clear favourite, which approach computes a better bound.
		Only with respect to the running time there is a slight advantage for using the sparse version.
	\item Especially for few terms, \algoFork{} runs significantly faster than \algoBnB{}.
		This speed-up partially comes from parallel computations, but also from our preprocessing, so we run fewer instances of SONC and SAGE.
	\item Computing the minimal orthants for \algoFork{} runs fast.
		So it is possible to a priori get an estimate of the running time of \algoFork{}.
		If the gain from eliminating orthant is too small, we can simply run \algoBnB{} instead.
\end{enumerate}

The most interesting course for further work is to find better cut criteria for the branch-and-bound approach.
These might significantly improve the running time of our algorithm.
Also, the order, in which we branch the variables, is important for the size of the search tree and thus, for the running time.
A significant advantage for \algoFork{} is its parallelisation.
To speed up \algoBnB{}, we can compute several nodes of the search tree in parallel.

To improve the results of \algoFork{}, we can identify the orthants, where some computation failed.
These orthants correspond to leaves in the search tree of \algoBnB{}.
By moving up in the tree and computing lower bounds for these nodes, we still obtain lower bounds for the original orthant.

Finally, we want to emphasise, that both \algoBnB{} and \algoFork{} are just frameworks, which use SONC and SAGE.
So any improvement for these, which could be quality of results, running time or numerical stability, results in an improvement for the algorithms in this paper.

\bibliographystyle{halphainit}
\bibliography{sonc_branch_and_bound}

\end{document}